\documentclass[12pt,reqno]{amsart}

\newtheorem{theorem}{Theorem}[section]
\newtheorem{proposition}[theorem]{Proposition}
\newtheorem{lemma}[theorem]{Lemma}
\newtheorem{corollary}[theorem]{Corollary}

\theoremstyle{definition}
\newtheorem{definition}[theorem]{Definition}

\numberwithin{equation}{section}

\begin{document}

\baselineskip=15.5pt

\title[On the generalized $\text{SO}(2n,{\mathbb C})$--opers]{On the generalized
$\text{SO}(2n,{\mathbb C})$--opers}

\author[I. Biswas]{Indranil Biswas$^{^*}$}

\address{School of Mathematics, Tata Institute of Fundamental
Research, Homi Bhabha Road, Mumbai 400005, India}

\email{indranil@math.tifr.res.in}

\author[L. P. Schaposnik]{Laura P. Schaposnik}

\address{Department of Mathematics, Statistics, and Computer Science, University of
Illinois at Chicago, 851 S Morgan St, Chicago, IL 60607, United States}

\email{schapos@uic.edu}

\author[ M. Yang]{Mengxue Yang}

\address{Department of Mathematics, Statistics, and Computer Science, University of
Illinois at Chicago, 851 S Morgan St, Chicago, IL 60607, United States}

\email{myang59@uic.edu}

\subjclass[2010]{14H60, 31A35, 33C80, 53C07}

\keywords{Oper, orthogonal group, differential operator, connection.}

\date{}

\begin{abstract}
Since their introduction by Beilinson-Drinfeld \cite{BD,Opers1}, opers have seen several generalizations. In 
\cite{BSY} a higher rank analog was studied, named {generalized $B$-opers}, where the successive quotients of the 
oper filtration are allowed to have higher rank and the underlying holomorphic vector bundle is endowed with a 
bilinear form which is compatible with both the filtration and the oper connection. Since the definition didn't 
encompass the even orthogonal groups, we dedicate this paper to study generalized $B$-opers whose structure group 
is ${\rm SO}(2n,\mathbb{C})$, and show their close relationship with geometric structures on a Riemann surface.
\end{abstract}

\maketitle

\section{Introduction}\label{intro.}

Motivated by the works of Drinfeld and Sokolov \cite{DS1,DS2}, Beilinson and Drinfeld introduced opers, in 
\cite{BD,Opers1}, for a semisimple complex Lie group $G$. A $G$--oper on a compact Riemann surface $X$ is
\begin{itemize}
\item a holomorphic principal $G$--bundle $P$ on $X$ equipped with a holomorphic connection $\nabla$, and

\item a holomorphic reduction of structure group of $P$ to a Borel subgroup of $G$,
\end{itemize}
such that the reduction satisfies the Griffiths transversality condition with respect to the connection $\nabla$
and the second fundamental form of $\nabla$ for the reduction satisfies
certain nondegeneracy conditions.

In recent years, different extensions of the above objects have been introduced and studied -- examples are 
$\mathfrak{g}$-opers (a $\mathfrak{g}$-oper is a $\text{Aut}(\mathfrak{g})$-oper \cite{Opers1}) and Miura opers 
\cite{MR2146349}, as well as $(G,P)$-opers \cite{GP}. Very recently, the authors introduced the notion of a {\it 
generalized $B$-oper} in \cite{BSY}. The definition was inspired by \cite{Bi}, where a particular class of opers 
was studied for which $\text{rank}(E)\,=\, nr$ and the rank of each successive quotient $E_i/E_{i-1}$ is $r$ (the 
above two conditions remain unchanged). In \cite{BSY}, the authors incorporated a non-degenerate bilinear form $B$ 
and required the (not necessarily full) filtration and the connection appearing in a $G$--oper to be compatible 
with it. However, the work done in \cite{BSY} did not apply to opers with structure group ${\rm SO}(2n,\mathbb{C})$.
The case of ${\rm SO}(2n,\mathbb{C})$ is subtler than ${\rm SO}(2n+1,\mathbb{C})$ and ${\rm Sp}(2n,\mathbb{C})$.
We dedicate the present paper to the study the ${\rm SO}(2n,\mathbb{C})$ case.

We begin our work by considering filtered $\text{SO}(2n,{\mathbb C})$--bundles with connections in Section 
\ref{filtered}, leading to the introduction and study of a \textit{generalized} $\text{SO}(2n, {\mathbb 
C})$--\textit{quasioper}: a quadruple $(E,\, B_0,\, {\mathcal F}_{\bullet},\, D)$, where $(E,\, B_0,\, {\mathcal 
F}_{\bullet})$ is a filtered $\text{SO}(2n, {\mathbb C})$--bundle over
a compact Riemann surface $X$, and $D$ is a holomorphic connection on $(E,\, 
B_0,\, {\mathcal F}_{\bullet})$ (see Definition \ref{def3}). Let $2m+1$ be the length of the filtration.
Then $r\,=\, n/(m+1)$ is an integer.

The quasiopers have naturally induced isomorphic 
dual quasiopers, as shown in Proposition \ref{dual}. The properties of generalized $\text{SO}(2n, {\mathbb 
C})$--quasiopers are studied in Section \ref{se3}, in the spirit of \cite{BSY} and in relation to the jet 
bundles.

The main goal of the paper is to introduce $\text{SO}(2n, {\mathbb C})$--\textit{opers}, and to show that 
generalized $\text{SO}(2n, {\mathbb C})$--opers are closely related to projective structures on the base 
Riemann surface $X$, and this is done in Section \ref{opers}. After constructing and studying $\text{SO}(2n, 
{\mathbb C})$--opers through $\text{SO}(2n, {\mathbb C})$--quasiopers, we consider their relation 
to geometric structures on $X$.

Let $X$ be a compact connected Riemann surfaces of genus at least two.
Fix positive integers $n$ and $m$ such that $r\,:=\,n/(m+1)$ is an integer. Let
$$
{\mathbb O}_X(n,m)
$$
denote the space of all isomorphism classes of generalized ${\rm SO}(2n,{\mathbb C})$--opers on $X$
of filtration length $2m+1$ (see Definition \ref{def1} and Definition \ref{def4}).
Let ${\mathcal C}_X$ be the space of all isomorphism classes of holomorphic 
principal $\text{SO}(r, {\mathbb C})$--bundles on $X$ equipped with a holomorphic
connection, and let ${\mathfrak P}(X)$ be the space of all projective structures on the Riemann surface $X$.

We prove the following (see Theorem \ref{thm1}):

\begin{theorem}\label{thi}
If the integer $r$ is odd, then 
there is a canonical bijection between ${\mathbb O}_X(n,m)$ and the Cartesian product
$$
{\mathcal C}_X\times {\mathfrak P}(X)\times \left(H^0(X,\, K^{\otimes (m+1)}_X)\oplus
\left(\bigoplus_{i=2}^m H^0(X,\, K^{\otimes 2i}_X)\right)\right)\times J(X)_2\ ,
$$
for $J(X)_2$ the group of holomorphic line bundles on $X$ of order two, and $K_X$ the
holomorphic cotangent bundle of $X$. 

If $r$ is even, then 
there is a canonical bijection between ${\mathbb O}_X(n,m)$ and the Cartesian product
$$
{\mathcal C}_X\times {\mathfrak P}(X)\times \left(H^0(X,\, K^{\otimes (m+1)}_X)\oplus
\left(\bigoplus_{i=2}^m H^0(X,\, K^{\otimes 2i}_X)\right)\right)\, .
$$
\end{theorem}

We note that in the cases of ${\rm SO}(2n+1,\mathbb{C})$ and ${\rm Sp}(2n,\mathbb{C})$, the
decomposition is same for even and odd $r$, unlike in Theorem \ref{thi}.

\section{Filtered $\text{SO}(2n,{\mathbb C})$-bundles with connections}\label{filtered}

Let $X$ be a compact connected Riemann surface of genus $g$, with $g\, \geq\, 2$.
The holomorphic cotangent bundle and the holomorphic tangent bundle of $X$ will be denoted by $K_X$
and $TX$ respectively.
Let $E$ be a holomorphic vector bundle on $X$ of rank $2n$, where $n\, \geq\, 2$,
such that $$\det E\,=\, \bigwedge\nolimits^{2n} E\,=\, {\mathcal O}_X\, $$ An $\text{SO}(2n,{\mathbb C})$
structure on $E$ is a holomorphic symmetric bilinear form
\begin{equation}\label{b0}
B_0\, \in\, H^0(X,\, \text{Sym}^2(E^*))
\end{equation}
on $E$ which is fiberwise nondegenerate. In other words, $B_0(x)$ is a nondegenerate
symmetric bilinear form on $E_x$ for every $x\, \in\, X$. A pair of the form $(E,\, B_0)$, where $B_0$ is an
$\text{SO}(2n,{\mathbb C})$ structure on a holomorphic vector bundle on $X$ $E$, would be called an $\text{SO}(2n,
{\mathbb C})$--bundle on $X$.

We note that for an $\text{SO}(2n, {\mathbb C})$--bundle $(E,\, B_0)$, the determinant line bundle
$\bigwedge^{2n} E$ is holomorphically identified with ${\mathcal O}_X$ uniquely up to a
sign. More precisely, for any $x\, \in\, X$, consider all isomorphisms of $(E_x, \, B_0(x))$ with
${\mathbb C}^{2n}$ equipped with the standard symmetric bilinear form. Then the space of corresponding
isomorphisms of $\bigwedge\nolimits^{2n} E_x$ with $\bigwedge\nolimits^{2n}{\mathbb C}^{2n}$ has exactly
two elements, and these two elements just differ by a sign.

\subsection{Filtered $\text{SO}(2n,{\mathbb C})$-bundles}

An $\text{SO}(2n,{\mathbb C})$ structure $B_0$ on $E$ produces a holomorphic isomorphism
\begin{equation}\label{e1}
B\, :\, E\, \longrightarrow\, E^*
\end{equation}
that sends any $v\, \in\, E_x$, $x\, \in\, X$, to the element of $E^*_x$ defined by
$w\, \longmapsto\, B_0(x)(w,\, v)$. The annihilator of a holomorphic subbundle $F\, \subset\, E$, for
the bilinear form $B_0$, will be denoted by $F^\perp$. So, for any $x\, \in\, X$, the subspace $F^\perp_x\,
\subset\, E_x$
consists of all $v\, \in\, E_x$ such that $B_0(x)(w,\, v)\, =\, 0$ for all $w\, \in\, F_x$. The bilinear form $B_0$
produces $C^\infty$ homomorphisms
$$
E\otimes (E\otimes \Omega{0,1}_X)\, \longrightarrow\, \Omega{0,1}_X\ \ \text{ and }\ \
(E\otimes \Omega{0,1}_X)\otimes E\, \longrightarrow\, \Omega{0,1}_X
$$
simply by tensoring with the identity map of $\Omega{0,1}_X$. Since
the bilinear form $B_0$ is holomorphic, we have
\begin{equation}\label{j1}
\overline{\partial} B_0(s,\, t)\,=\, B_0(\overline{\partial}_E s,\, t)+ B_0(s,\, \overline{\partial}_E t)\, ,
\end{equation}
where $s$ and $t$ are locally defined $C^\infty$ sections of $E$ and $\overline{\partial}_E\, :\,
C^\infty(X,\, E)\, \longrightarrow\, C^\infty(X,\, E\otimes \Omega{0,1}_X)$ is
the Dolbeault operator defining the holomorphic structure on $E$. If $t$ is a locally defined holomorphic 
section of $F$ and $s$ is a locally defined $C^\infty$ section of $F^\perp$, then from \eqref{j1} we have
$$
B_0(\overline{\partial}_E s,\, t)\,=\, 0\, ,
$$
because $\overline{\partial}_E t\,=\, 0\, B_0(s,\, t)$.
This implies that $F^\perp$ is actually a holomorphic subbundle of $E$; its rank is $2n-\text{rank}(F)$.

\begin{definition}
A \textit{filtration} of an $\text{SO}(2n, {\mathbb C})$--bundle $(E,\, B_0)$ is a filtration
of holomorphic subbundles of $E$
\begin{equation}\label{e-1}
0\, =\, F_0\, \subset\, F_1\, \subset\, F_2\, \subset \, \cdots\, \subset \, F_i\, \subset
\, \cdots\, \subset \, F_{2m} \, \subset \, F_{2m+1}\, =\, E
\end{equation}
satisfying the following two conditions:
\begin{enumerate}
\item $\text{rank}(F_{m+1}/F_m)\, =\, 2\cdot \text{rank}(F_1)$, and
$\text{rank}(F_i/F_{i-1})\, =\, \text{rank}(F_1)$, \\ for all $i\, \in\, \{1,\, \cdots,\, 2m+1\}
\setminus \{m+1\}$, and

\item $F^\perp_i\, =\, F_{2m+1-i}$ for all $0\, \leq\, i\, \leq\, m$.
\end{enumerate}
\end{definition}

Note that the first condition implies that $$(m+1)\cdot \text{rank}(F_1)\,=\, n\, .$$
The second condition implies that
the restriction $B_0\vert_{F_{m+1}}$ of the form $B_0$ to the subbundle $F_{m+1}\, \subset\, E$
has has the following two properties:
\begin{itemize}
\item the subbundle $F_{m}\, \subset\, F_{m+1}$ is annihilated by $B_0\vert_{F_{m+1}}$, meaning
$F_{m}\, \subset\, F^\perp_m$, and

\item the restriction $B_0\big\vert_{F_{m+1}}$ descends to the quotient bundle $F_{m+1}/F_m$ as a
fiberwise nondegenerate symmetric bilinear form.
\end{itemize}

For notational convenience, the filtration $\{F_i\}_{i=0}^{2m+1}$ in \eqref{e-1} will henceforth be
denoted by ${\mathcal F}_{\bullet}$.

\begin{definition}\label{def1}
An $\text{SO}(2n, {\mathbb C})$--bundle equipped with a filtration will be called
a \textit{filtered} $\text{SO}(2n, {\mathbb C})$--bundle. The odd integer $2m+1$ in
\eqref{e-1} will be called the \textit{length} of the filtration.
\end{definition}

We shall always assume that $m\, \geq\, 2$. This is because $$\text{SO}(4,{\mathbb C})\,=\,
(\text{SL}(2,{\mathbb C})\times\text{SL}(2,{\mathbb C}))/({\mathbb Z}/2{\mathbb Z})\, .$$

\subsection{$\text{SO}(2n,{\mathbb C})$-quasiopers: Filtered $\text{SO}(2n,{\mathbb C})$-bundles with connections}

Recall that a \textit{holomorphic connection} on a holomorphic vector bundle $E$ on $X$ is a first order
holomorphic differential operator 
$$
D\, :\, E\, \longrightarrow\, E\otimes K_X
$$
satisfying the Leibniz identity, this is, \[D(fs)\,=\, fD(s)+ s\otimes df\] for
any locally defined holomorphic function $f$ on $X$ and any locally defined holomorphic section $s$
of $E$ \cite{At}. In particular, a holomorphic connection is automatically flat because $\Omega^{2,0}_X\,=\, 0$.
The bilinear form $B_0$ in \eqref{b0} produces holomorphic homomorphisms
$$
E\otimes (E\otimes K_X)\, \longrightarrow\, K_X\ \ \text{ and }\ \
(E\otimes K_X)\otimes E\, \longrightarrow\, K_X
$$
simply by tensoring with the identity map of $K_X$.
A holomorphic connection on an $\text{SO}(2n,
{\mathbb C})$--bundle $(E,\, B_0)$ is a holomorphic connection $D$ on the holomorphic vector bundle
$E$ satisfying the identity
$$
\partial B_0(s,\, t)\,=\, B_0(D(s),\, t) + B_0(s,\, D(t))
$$
for all locally defined holomorphic sections $s$ and $t$ of $E$.

We note that for a holomorphic connection $D$ on an $\text{SO}(2n,
{\mathbb C})$--bundle $(E,\, B_0)$, the connection
on the determinant line bundle $\bigwedge^{2n} E\,=\, {\mathcal O}_X$ induced by $D$ coincides with the
trivial connection on the trivial holomorphic line bundle given by the de Rham differential $d$
(it is the unique rank one holomorphic connection on $X$ with trivial monodromy). Indeed, this
follows immediately from the fact that the isomorphism $B$ in \eqref{e1} takes the connection $D$ on
$E$ to the dual connection on $E^*$ induced by $D$.

Let $D$ a holomorphic connection on $E$, and let
\[
F_1\, \subset\, F_2\, \subset\, E\ \ \text{ and }\ \
F_3 \, \subset\, F_4\, \subset\, E
\]
be holomorphic subbundles such that $$D(F_1)\, \subset\, F_3\otimes K_X\ \ \text{ and }\ \
D(F_2)\, \subset\, F_4\otimes K_X\, .$$ 

\begin{definition}\label{second}The {\rm second fundamental form} of
$(F_1,F_2,F_3,F_4)$ for the connection $D$ is the map
\begin{eqnarray}\label{en}
S(D; F_1,F_2,F_3, F_4) \, :\, F_2/F_1 & \longrightarrow & (F_4/F_3)\otimes K_X\\
s& \longmapsto & D(\widetilde{s})\nonumber
\end{eqnarray}
that sends any locally defined holomorphic section $s$ of $F_2/F_1$ to the image
of $D(\widetilde{s})$ in $(F_4/F_3)\otimes K_X$, where $\widetilde{s}$ is any
locally defined holomorphic section of the subbundle $F_2$ that projects to $s$
under the quotient map $F_2\, \longrightarrow\, F_2/F_1$.
\end{definition}

It is straightforward to check that the image
of $D(\widetilde{s})$ in $(F_4/F_3)\otimes K_X$ does not depend on the
choice of the above lift $\widetilde{s}$ of $s$ (see \cite[Lemma 2.10]{BSY}).

\begin{definition}\label{def2}
Let $(E,\, B_0,\, {\mathcal F}_{\bullet})$ be a filtered $\text{SO}(2n, {\mathbb C})$--bundle.
A {\it holomorphic connection on $(E,\, B_0,\, {\mathcal F}_{\bullet})$} is a holomorphic connection
$D$ on $(E,\, B_0)$ satisfying the following three conditions:
\begin{enumerate}
\item $D(F_i)\, \subset\, F_{i+1}\otimes K_X$ for all $1\, \leq\, i\, \leq\, 2m$ (see \eqref{e-1}),

\item the second fundamental form
$$
S(D,i)\, :\, F_i/F_{i-1}\, \longrightarrow\, (F_{i+1}/F_i)\otimes K_X
$$
is an isomorphism for all $i\, \in\, \{1,\, \cdots, \,2m+1\}\setminus\{m,\,m+1\}$, and

\item the composition of homomorphisms
$$
(S(D,m+1)\otimes\text{Id}_{K_X})\circ S(D,m)\, :\, F_m/F_{m-1}
\, \longrightarrow\, (F_{m+2}/F_{m+1})\otimes K^{\otimes 2}_X
$$
is an isomorphism.
\end{enumerate}
\end{definition}

\begin{definition}\label{def3}
A \textit{generalized} $\text{SO}(2n, {\mathbb C})$--\textit{quasioper} on $X$ is
a quadruple $(E,\, B_0,\, {\mathcal F}_{\bullet},\, D)$, where
$(E,\, B_0,\, {\mathcal F}_{\bullet})$ is a filtered $\text{SO}(2n, {\mathbb C})$--bundle,
and $D$ is a holomorphic connection on the filtered $\text{SO}(2n, {\mathbb C})$--bundle
$(E,\, B_0,\, {\mathcal F}_{\bullet})$.

Two generalized $\text{SO}(2n, {\mathbb C})$--quasiopers $(E,\, B_0,\, {\mathcal F}_{\bullet},\, D)$
and $(E',\, B'_0,\, {\mathcal F}'_{\bullet},\, D')$ are called \textit{isomorphic} if there is a holomorphic
isomorphism of vector bundles
$$
\Phi\,:\, E\, \longrightarrow\, E'
$$
such that
\begin{itemize}
\item $\Phi$ takes the bilinear form $B_0$ on $E$ to the bilinear form $B'_0$ on $E'$,

\item $\Phi$ takes the filtration ${\mathcal F}_{\bullet}$ of $E$ to the filtration
${\mathcal F}'_{\bullet}$ of $E'$, and

\item $\Phi$ takes the connection $\nabla$ on $E$ to the connection $\nabla'$ on $E'$.
\end{itemize}
\end{definition}

\begin{proposition}\label{dual}
Given a generalized ${\rm SO}(2n, {\mathbb C})$--quasioper
$(E,\, B_0,\, {\mathcal F}_{\bullet},\, D)$, there is a naturally associated isomorphic dual quasioper. 
\end{proposition}

\begin{proof}
Let $(E,\, B_0,\, {\mathcal F}_{\bullet},\, D)$ be a generalized $\text{SO}(2n, {\mathbb C})$--quasioper
on $X$, where ${\mathcal F}_{\bullet}$, as in \eqref{e-1}, is a filtration of $E$. Consider the dual
vector bundle $E^*$. It is equipped with a holomorphic connection $D^*$ induced by the connection
$D$ on $E$.

Since the symmetric bilinear form $B_0$ on $E$ is nondegenerate, it produces a holomorphic
symmetric nondegenerate bilinear form $B^*_0$ on $E^*$. For any $F_i$ in \eqref{e-1}, define
\begin{equation}\label{gi}
G_{2m+1-i}\, \subset\, E^*
\end{equation}
to be the kernel of the natural projection $E^*\, \longrightarrow\, (F_i)^*$. Then
\begin{equation}\label{z}
(E^*,\, B^*_0, \{G_j\}_{j=0}^{2m+1},\, D^*)
\end{equation}
is also a generalized $\text{SO}(2n, {\mathbb C})$--quasioper.

It is straightforward to check that the holomorphic isomorphism $B$ in \eqref{e1}
takes the generalized $\text{SO}(2n, {\mathbb C})$--quasioper $(E,\, B_0,\, {\mathcal F}_{\bullet},\, D)$
to the generalized $\text{SO}(2n, {\mathbb C})$--quasioper $(E^*,\, B^*_0, \{G_j\}_{j=0}^{2m+1},\, D^*)$
constructed in \eqref{z}.
\end{proof}

\section{Properties of a generalized $\text{SO}(2n, {\mathbb C})$--quasioper}\label{se3}

Let $W$ be holomorphic vector bundle over $X$ equipped with a holomorphic connection $D_W$, and 
let $V\, \subset\, W$ be any holomorphic subbundle.

\begin{lemma}\label{3.1}
There is a unique minimal holomorphic subbundle $\widehat{D}_W(V)$
of $W$ containing $V$ such that the connection $D_W$ takes $V$ into $\widehat{D}_W(V)\otimes K_X$.
\end{lemma}

\begin{proof} From Definition \ref{second}, consider the second fundamental form of $V$
for the connection $D_W$
$$
S(D_W; V)\, :\, V\, \longrightarrow\, (W/V)\otimes K_X
$$
by letting $F_1\,=\, 0, F_2\,=\, F_3\,=\, V$
and $F_4\,=\, W$ in Eq.~\eqref{en}. Let
$${\mathcal T}\, \subset\, ((W/V)\otimes K_X)/(S(D_W; V)(V))$$ be the torsion part of the coherent
analytic sheaf $((W/V)\otimes K_X)/(S(D_W; V)(V))$. The
inverse image of $\mathcal T$ under the quotient map \[(W/V)\otimes K_X\, \longrightarrow\,
((W/V)\otimes K_X)/(S(D; V)(V))\] will be denoted by ${\mathcal F}$. So
${\mathcal F}\otimes TX$ is a holomorphic subbundle of $$(W/V)\otimes K_X\otimes TX\,=\, W/V\,.$$

The inverse image of the subbundle
${\mathcal F}\otimes TX\, \subset\, W/V$ under the quotient map $W\, \longrightarrow\, W/V$
will be denoted by $\widehat{D}_W(V)$.

Note that $\widehat{D}_W(V)$ is a holomorphic subbundle of $W$, because ${\mathcal F}\otimes TX$
is a holomorphic subbundle of $W/V$. Also, $V$ is a holomorphic subbundle of $\widehat{D}_W(V)$. From
the construction of $\widehat{D}_W(V)$ it is evident that we have
$$D_W (V)\, \subset\, \widehat{D}_W(V)\otimes K_X\, .$$ Also it is clear that
$\widehat{D}_W(V)$ is the smallest among all subbundles $U$ of $W$ such that
$D_W (V)\, \subset\, U\otimes K_X$.
\end{proof}

Note that $V$ is preserved by the connection $D_W$ if and only if $\widehat{D}_W(V)
\,=\, V$, where $\widehat{D}_W(V)$ is constructed in Lemma \ref{3.1}.

The holomorphic subbundle $\widehat{D}_W(\widehat{D}_W(V))\, \subset \, W$ will be denoted by
$\widehat{D}^2_W(V)$. Moreover, for ease of notation, inductively define the subbundle
\begin{equation}\label{indd}
\widehat{D}^{k+1}_W(V)\, :=\, \widehat{D}_W(\widehat{D}^k_W(V))\, \subset \, W\, ,
\end{equation}
for $k \,\geq\, 2$. So $\{\widehat{D}^{j}_W(V)\}_{j\geq 1}$ is an increasing sequence of
holomorphic subbundles of $W$.

Through Lemma \ref{3.1}, we can construct a holomorphic subbundle of a generalized ${\rm SO}(2n,\mathbb{C})$--quasioper.
Indeed, let $$(E,\, B_0,\, {\mathcal F}_{\bullet},\, D)$$ be a generalized $\text{SO}(2n, {\mathbb C})$--quasioper
on $X$, where ${\mathcal F}_{\bullet}$, as in \eqref{e-1}, is a filtration
\begin{equation}\label{e-1b}
0\, =\, F_0\, \subset\, F_1\, \subset\, F_2 \,\subset\, \cdots\, \subset \, F_i\, \subset
\, \cdots\, \subset \, F_{2m} \, \subset \, F_{2m+1}\, =\, E
\end{equation}
of length $2m+1$.
For the holomorphic subbundle $F_1\, \subset\, E$ in \eqref{e-1b}, define the holomorphic subbundle
\begin{equation}\label{e6}
{\mathbb F}\, :=\, \widehat{D}^{2m}_E(F_1) \, \subset\, E
\end{equation}
(see \eqref{indd}).
We note that the subbundle ${\mathbb F}$ in general is not preserved by the connection $D$ on $E$.

Now consider the generalized $\text{SO}(2n, {\mathbb C})$--quasioper
$(E^*,\, B^*_0, \{G_j\}_{j=0}^{2m+1},\, D^*)$ in \eqref{z} associated to
$(E,\, B_0,\, {\mathcal F}_{\bullet},\, D)$ via Proposition \ref{dual}. As in \eqref{e6}, define
the holomorphic subbundle
\begin{equation}\label{eg}
{\mathbb G}\, :=\, (\widehat{D}^*)^{2m}_{E^*}(G_1) \, \subset\, E^*\, ,
\end{equation}
where $G_1$ is constructed in \eqref{gi}. The dual of the natural quotient map
$E^*\, \longrightarrow\, E^*/{\mathbb G}$ is a fiberwise injective holomorphic homomorphism
$(E^*/{\mathbb G})^*\, \longrightarrow\, E^{**}$.
Therefore, we have the holomorphic subbundle
\begin{equation}\label{s}
{\mathcal S}\, :=\, (E^*/{\mathbb G})^* \, \subset\, E^{**}\,=\, E
\end{equation}
given by the image of the above fiberwise injective homomorphism.

\begin{lemma}\label{lem1}
For the holomorphic subbundles $\mathbb F$ and ${\mathcal S}$ of $E$, in \eqref{e6} and \eqref{s}
respectively, the natural homomorphism
$$
{\mathbb F}\oplus{\mathcal S}\, \longrightarrow\, E
$$
is an isomorphism.
Moreover, the resulting holomorphic decomposition $$E\,=\, {\mathbb F}\oplus{\mathcal S}$$ of $E$
is orthogonal with respect to the bilinear form $B_0$ on $E$.
\end{lemma}

\begin{proof}
From the properties of the filtration ${\mathcal F}_{\bullet}$ and the connection $D$ it follows that
the natural homomorphism
$$
{\mathbb F}\oplus{\mathcal S}\, \longrightarrow\, E
$$
is surjective. Note that from the properties of the filtration ${\mathcal F}_{\bullet}$ and $D$
we also have $\text{rank}({\mathbb F})\,=\, 2n-\frac{n}{m+1}\,=\, \text{rank}(E)-\text{rank}(F_1)$
and $\text{rank}({\mathcal S})\,=\, \frac{n}{m+1}\,=\, \text{rank}(F_1)$.
Furthermore, we have $B_0({\mathbb F},\, {\mathcal S})\,=\, 0$. These together imply that
${\mathbb F}\bigcap {\mathcal S}\,=\, 0$ and ${\mathbb F}^\perp\,=\, {\mathcal S}$.
\end{proof}

In what follows we shall describe an alternative construction of the subbundle $\mathcal S$ in \eqref{s} by
considering the jet bundle approach given in \cite{BSY}. Let
\begin{equation}\label{e2}
Q\, :=\, E/F_{2m}
\end{equation}
be the quotient in \eqref{e-1b}, and let
\begin{equation}\label{e3}
q\, :\, E\, \longrightarrow\, E/F_{2m}\,=\, Q
\end{equation}
be the quotient map.

For any nonnegative integer $i$, let $J^i(Q)$ be the $i$-th order jet bundle of $Q$ in \eqref{e2}
(see \cite[Section 3.1]{BSY}, \cite{Bi}, \cite{Bi1} for jet bundles). As shown in \cite[Eq. (3.3)]{BSY},
\cite[Eq. (3.5)]{BSY}, the connection $D$ on $E$
produces an ${\mathcal O}_X$--linear homomorphism 
\begin{equation}\label{e4}
f_i\, :\, E\, \longrightarrow\, J^i(Q)\, .
\end{equation}
We briefly recall the construction of $f_i$ as this homomorphism plays a crucial role.

Take any point $x\, \in\, X$, and let $x\, \in\, U_x\, \subset\, X$ be a simply connected
analytic open neighborhood of the point $x$. For any $v\, \in\, E_x$, let $\widetilde{v}$ be the
unique flat section of $E\big\vert_{U_x}$, for the flat connection $D$ on $E$, such that $\widetilde{v}(x)\,=\, v$. 
Consider the holomorphic section $q(\widetilde{v})$ of the vector bundle $Q\vert_{U_x}$ in \eqref{e2},
where $q$ is the projection in \eqref{e3}. Restricting this section $q(\widetilde{v})$ to the $i$--th order
infinitesimal neighborhood of $x$, we get an element of $J^i(Q)_x$; this element of
$J^i(Q)_x$ given by $q(\widetilde{v})$ will be denoted by $q(\widetilde{v})_i$. The map $f_i$
in \eqref{e4} sends any $v\, \in\, E_x$, $x\, \in\, X$, to $q(\widetilde{v})_i\,\in\,
J^i(Q)_x$ constructed above using $v$ and the connection $D$.

{}From the three conditions in Definition \ref{def2} it follows that the homomorphism
$$
f_{2m}\, :\, E\, \longrightarrow\, J^{2m}(Q)
$$
in \eqref{e4} is surjective. Moreover, the subbundle $\mathcal S$ in \eqref{s} coincides
with the kernel of the above homomorphism $f_{2m}$.
Consequently, we have a short exact sequence of holomorphic vector bundles
\begin{equation}\label{e5}
0\, \longrightarrow\, {\mathcal S}\, =\, \text{kernel}(f_{2m})
\, \longrightarrow\, E \,\stackrel{f_{2m}}{\longrightarrow}
\, J^{2m}(Q) \, \longrightarrow\, 0
\end{equation}
on $X$. Therefore, Lemma \ref{lem1} has the following corollary.

\begin{corollary}\label{cor1}
The composition of homomorphisms
$$
{\mathbb F}\, \hookrightarrow\, E\, \stackrel{f_{2m}}{\longrightarrow}\, J^{2m}(Q)\, ,
$$
where $f_{2m}$ is the homomorphism in \eqref{e4}, and $\mathbb F$ is the subbundle
of $E$ in Lemma \ref{lem1}, is an isomorphism.
\end{corollary}

Let
\begin{equation}\label{ch2}
f'_{2m}\, :\, {\mathbb F}\, \stackrel{\sim}{\longrightarrow}\, J^{2m}(Q)
\end{equation}
be the composition of homomorphisms in Corollary \ref{cor1}; recall from Corollary \ref{cor1}
that $f'_{2m}$ is an isomorphism.

Using the decomposition ${\mathbb F}\oplus{\mathcal S}\,=\, E$ in Lemma \ref{lem1}, consider the
composition of homomorphisms
\begin{equation}\label{co1}
{\mathbb F} \, \hookrightarrow\, E \, \stackrel{D}{\longrightarrow}\, E\otimes K_X
\, \stackrel{q_{\mathbb F}\otimes \text{Id}}{\longrightarrow}\,{\mathbb F}\otimes K_X\, ,
\end{equation}
where $q_{\mathbb F}\, :\, E\,=\, {\mathbb F}\oplus{\mathcal S}\, \longrightarrow\, {\mathbb F}$ is
the natural projection to factor ${\mathbb F}$.
This composition of homomorphisms is a holomorphic connection on ${\mathbb F}$, because
it satisfies the Leibniz identity. The holomorphic
connection on $\mathbb F$ constructed in \eqref{co1} will be denoted by $D^{\mathbb F}$.

Similarly, the composition of homomorphisms
\begin{equation}\label{co2}
{\mathcal S} \, \hookrightarrow\, E \, \stackrel{D}{\longrightarrow}\, E\otimes K_X
\, \stackrel{q_{\mathcal S}\otimes \text{Id}}{\longrightarrow}\,{\mathcal S}\otimes K_X\, ,
\end{equation}
where $q_{\mathcal S}\, :\, E\,=\, {\mathbb F}\oplus{\mathcal S}\, \longrightarrow\, {\mathcal S}$ is
the natural projection to factor ${\mathcal S}$
in Lemma \ref{lem1}, is a holomorphic connection on the holomorphic vector bundle ${\mathcal S}$. The holomorphic
connection on $\mathcal S$ constructed in \eqref{co2} will be denoted by $D^{\mathcal S}$.

The holomorphic connections $D^{\mathbb F}$ and $D^{\mathcal S}$, on $\mathbb F$ and
$\mathcal S$ respectively, together define a holomorphic connection $D^{\mathbb F}\oplus D^{\mathcal S}$
on ${\mathbb F}\oplus{\mathcal S}$.
It should be emphasized that the isomorphism ${\mathbb F}\oplus{\mathcal S}\,=\, E$ in Lemma \ref{lem1}
does not, in general, take the holomorphic connection $D^{\mathbb F}\oplus D^{\mathcal S}$ on
${\mathbb F}\oplus{\mathcal S}$ to the connection $D$ on $E$. Indeed, for the connection $D$
on $E$ the subbundles ${\mathbb F}$ and ${\mathcal S}$ of $E$ may have
nontrivial second fundamental form. On the
other hand, for the direct sum of connections $D^{\mathbb F}\oplus D^{\mathcal S}$ the
second fundamental form of both ${\mathbb F}$ and ${\mathcal S}$ vanish identically.

From Lemma \ref{lem1} it follows immediately that the restrictions of the bilinear form $B_0$ on $E$ to
both $\mathbb F$ and $\mathcal S$ are nondegenerate. The holomorphic symmetric nondegenerate
bilinear form on $\mathbb F$ (respectively, $\mathcal S$) obtained by restricting $B_0$
to $\mathbb F$ (respectively, $\mathcal S$) will be
denoted by $B_{\mathbb F}$ (respectively, $B_{\mathcal S}$); in particular, we have
$$
B_{\mathbb F}\, \in\, H^0(X,\, \text{Sym}^2({\mathbb F}^*))\ \ \text{ and }\ \
B_{\mathbb S}\, \in\, H^0(X,\, \text{Sym}^2({\mathbb S}^*))\, .
$$
As the
decomposition ${\mathbb F}\oplus{\mathcal S}\,=\, E$ in Lemma \ref{lem1} is orthogonal, we have
\begin{equation}\label{B}
B_0\,=\, B_{\mathbb F}\oplus B_{\mathcal S}\, .
\end{equation}

Since the connection $D$ preserves the bilinear form $B_0$ on $E$, and \eqref{B} holds, from the
constructions of the connections $D^{\mathbb F}$ in \eqref{co1} and the connection $D^{\mathcal S}$
in \eqref{co2} we have the following:

\begin{corollary}\label{cor2}
The connection $D^{\mathbb F}$ on $\mathbb F$ in \eqref{co1} preserves the bilinear form $B_{\mathbb F}$
on $\mathbb F$.

The connection $D^{\mathcal S}$ on $\mathcal S$ in \eqref{co2} preserves the bilinear form $B_{\mathcal S}$
on $\mathcal S$.
\end{corollary}

\begin{proposition}\label{dq}
The connection $D^{\mathbb F}$ on $\mathbb F$ produces a
holomorphic connection $D_Q$ on $J^{2m}(Q)$.
\end{proposition}

\begin{proof}
This can be deduced from the fact that the homomorphism $f'_{2m}$ in \eqref{ch2}
is an isomorphism. So $D_Q$ is the holomorphic connection on $J^{2m}(Q)$ that corresponds to the
connection $D^{\mathbb F}$ on $\mathbb F$ by this isomorphism $f'_{2m}$.

We shall give a direct construction of this connection $D_Q$ on $J^{2m}(Q)$.
Let
$$
0\, \longrightarrow\, Q\otimes K^{\otimes (2m+1)}_X\, \stackrel{\iota}{\longrightarrow}\,
J^{2m+1}(Q)\, \stackrel{q}{\longrightarrow}\, J^{2m}(Q)\, \longrightarrow\, 0
$$
be the natural short exact sequence of jet bundles. It fits in the following commutative
diagram of homomorphisms:
\begin{equation}\label{cd}
\begin{matrix}
&& 0 && 0\\
&& \Big\downarrow && \Big\downarrow\\
0 & \longrightarrow & Q\otimes K^{\otimes (2m+1)}_X & \stackrel{\iota}{\longrightarrow} &
J^{2m+1}(Q) & \stackrel{q}{\longrightarrow} & J^{2m}(Q) & \longrightarrow && 0\\
&&\Big\downarrow && ~\,~ \Big\downarrow \lambda && \Vert\\
0 & \longrightarrow & J^{2m}(Q)\otimes K_X & \stackrel{\iota'}{\longrightarrow} &
J^1(J^{2m}(Q)) & \stackrel{q'}{\longrightarrow} & J^{2m}(Q) & \longrightarrow && 0\\
&& \Big\downarrow &&\Big\downarrow \\
&& J^{2m-1}(Q)\otimes K_X & \stackrel{=}{\longrightarrow} & J^{2m-1}(Q)\otimes K_X\\
&& \Big\downarrow && \Big\downarrow\\
&& 0 && 0
\end{matrix}
\end{equation}
(see \cite[p.~4, (2.4)]{Bi} and \cite[p.~10, (3.4)]{Bi}).

Consider the homomorphism
$$
(f_{2m+1}\vert_{\mathbb F})\circ (f'_{2m})^{-1}\, :\, J^{2m}(Q)\, \longrightarrow\, J^{2m+1}(Q)\, ,
$$
where $f_{2m+1}\vert_{\mathbb F}$ is the restriction of the homomorphism in \eqref{e4},
and $f'_{2m}$ is the isomorphism in \eqref{ch2}. It is straightforward to check that
$$
q\circ ((f_{2m+1}\vert_{\mathbb F})\circ (f'_{2m})^{-1})\,=\, \text{Id}_{J^{2m}(Q)}\, ,
$$
where $q$ is the projection in \eqref{cd}. Therefore, from the commutativity of the
diagram in \eqref{cd} we conclude that
$$
q'\circ \lambda\circ ((f_{2m+1}\vert_{\mathbb F})\circ (f'_{2m})^{-1})\,=\, \text{Id}_{J^{2m}(Q)}\, ,
$$
where $q'$ and $\lambda$ are the homomorphisms in \eqref{cd}. Consequently, the homomorphism
$$
\lambda\circ ((f_{2m+1}\vert_{\mathbb F})\circ (f'_{2m})^{-1})\,:\, J^{2m}(Q)\, \longrightarrow\,
J^1(J^{2m}(Q))
$$
produces a holomorphic splitting of the short exact sequence
\begin{equation}\label{ecd}
0 \, \longrightarrow \, J^{2m}(Q)\otimes K_X \, \stackrel{\iota'}{\longrightarrow} \,
J^1(J^{2m}(Q))\, \stackrel{q'}{\longrightarrow} \, J^{2m}(Q)\, \longrightarrow\, 0
\end{equation}
in \eqref{cd}. But \eqref{ecd} is the twisted dual of the Atiyah exact sequence for $J^{2m}(Q)$. More
precisely, let
$$
0 \, \longrightarrow \, J^{2m}(Q)^*\otimes J^{2m}(Q) \, \longrightarrow \,J^1(J^{2m}(Q))^*\otimes J^{2m}(Q)
$$
$$
\stackrel{\eta}{\longrightarrow} \,
(J^{2m}(Q)\otimes K_X)^*\otimes J^{2m}(Q)\,=\, \text{End}(J^{2m}(Q))\otimes TX\, \longrightarrow \, 0
$$
be the dual of \eqref{ecd} tensored with $\text{Id}_{J^{2m}(Q)}$. Then $\eta^{-1}(\text{Id}_{J^{2m}(Q)}\otimes TX)
\, \subset\, J^1(J^{2m}(Q))^*\otimes J^{2m}(Q)$ is the Atiyah bundle $\text{At}(J^{2m}(Q))$ of
$J^{2m}(Q)$; furthermore, the short exact sequence
$$
0 \, \longrightarrow \, \text{End}(J^{2m}(Q))\,=\,
J^{2m}(Q)^*\otimes J^{2m}(Q) \, \longrightarrow \,\text{At}(J^{2m}(Q))
\, \stackrel{\eta}{\longrightarrow} \,TX \, \longrightarrow \,0
$$
obtained from the above short exact sequence is in fact the Atiyah exact sequence for $J^{2m}(Q)$.
Consequently, a holomorphic splitting of \eqref{ecd}
is a holomorphic connection on the holomorphic vector bundle $J^{2m}(Q)$
\cite{At}. Therefore, the above homomorphism $\lambda\circ ((f_{2m+1}\vert_{\mathbb F})\circ
(f'_{2m})^{-1})$ is a holomorphic connection on $J^{2m}(Q)$.

The holomorphic connection on $J^{2m}(Q)$
defined by $\lambda\circ ((f_{2m+1}\vert_{\mathbb F})\circ (f'_{2m})^{-1})$
coincides with the holomorphic connection $D_Q$ on $J^{2m}(Q)$ produced by
the connection $D^{\mathbb F}$ on $\mathbb F$ (see \eqref{co1}) using the isomorphism $f'_{2m}$ in \eqref{ch2}.
\end{proof}

Let ${\mathcal L}$ be holomorphic line bundle on $X$ of order two. So the holomorphic line bundle ${\mathcal L}
\otimes {\mathcal L}$ is holomorphically isomorphic to ${\mathcal O}_X$. Fix a holomorphic isomorphism
\begin{equation}\label{rho}
\rho\, :\, {\mathcal L}\otimes {\mathcal L}\, \longrightarrow\, {\mathcal O}_X\, .
\end{equation}
There is a unique holomorphic connection
\begin{equation}\label{dl}
D^{\mathcal L}
\end{equation}
on ${\mathcal L}$ such that the isomorphism
$\rho$ in \eqref{rho} takes the holomorphic connection $D^{\mathcal L}\otimes \text{Id}+\text{Id}\otimes
D^{\mathcal L}$ on ${\mathcal L}\otimes {\mathcal L}$ to the trivial connection on ${\mathcal O}_X$
given by the de Rham differential $d$. It should be clarified that this connection $D^{\mathcal L}$ does
not depend on the choice of the isomorphism $\rho$.

Let $(E,\, B_0,\, {\mathcal F}_{\bullet},\, D)$ be a generalized $\text{SO}(2n, {\mathbb C})$--quasioper
on $X$. Consider the holomorphic vector bundle $E^1\,:=\, E\otimes{\mathcal L}$. Note that
$$\bigwedge\nolimits^{2n}E^1\,=\, (\bigwedge\nolimits^{2n}E)\otimes (\bigwedge\nolimits^{2n}{\mathcal L})\,=\,
\bigwedge\nolimits^{2n}E\,=\, {\mathcal O}_X\, .$$
Since $E^1\otimes E^1\,=\, (E\otimes E)\otimes ({\mathcal L}\otimes {\mathcal L})$, we conclude that
$$B^1_0\,=\, B_0 \otimes\rho$$
is a fiberwise nondegenerate symmetric holomorphic bilinear form on $E^1$, where $\rho$ is the
isomorphism in \eqref{rho}. The filtration ${\mathcal F}_{\bullet}$ of holomorphic subbundles of $E$ produces a
filtration ${\mathcal F}^1_{\bullet}$ of holomorphic subbundles of $E^1$. The $i$--th term $F^1_i$ of
${\mathcal F}^1_{\bullet}$ is simply $F_i\otimes{\mathcal L}$ (see \eqref{e-1b}). Let
\begin{equation}\label{dl2}
D^1\, :=\, D\otimes\text{Id}_{\mathcal L}+ \text{Id}_E\otimes D^{\mathcal L}
\end{equation}
be the holomorphic connection on $E\otimes{\mathcal L}\,=\, E^1$, where $D^{\mathcal L}$ is the
holomorphic connection in \eqref{dl}.

The following lemma is straightforward to prove.

\begin{lemma}\label{lemt}
The quadruple $$(E^1,\, B^1_0,\, {\mathcal F}^1_{\bullet},\, D^1)$$ constructed above is
a generalized ${\rm SO}(2n, {\mathbb C})$--quasioper on $X$.
\end{lemma}

The holomorphic vector bundle ${\mathbb F}\,=\, \widehat{D}^{2m}_E(F_1)$ in \eqref{e6}
has the following filtration of holomorphic subbundles 
\begin{equation}\label{e7}
0\, \, \subset\, F_1\, \subset\, \widehat{D}_E(F_1)\, \subset\, \widehat{D}^2_E(F_1)\,
\subset\, \cdots\, \subset\, \widehat{D}^{2m-1}_E(F_1)\, \subset\, \widehat{D}^{2m}_E(F_1)
\,=\, {\mathbb F}\, .
\end{equation}
From Definition \ref{def2} it follows that the filtration of $\mathbb F$ in \eqref{e7}
coincides with the filtration of $\mathbb F$ obtained by intersecting the filtration
of $E$ in \eqref{e-1b} with the subbundle ${\mathbb F}$ of $E$. Moreover, the isomorphism
$f'_{2m}$ in \eqref{ch2} takes the
filtration of $\mathbb F$ in \eqref{e7} to the filtration of $J^{2m}(Q)$ given by the short exact
sequence of jet bundles
\begin{equation}\label{js}
0\, \longrightarrow\, Q\otimes K^{\otimes i}_X\, \longrightarrow\,
J^{i}(Q)\, \longrightarrow\, J^{i-1}(Q)\, \longrightarrow\, 0
\end{equation}
for $i\, \geq\, 1$.
More precisely, for any $1\, \leq\, j\, \leq\, 2m-1$, the subbundle $\widehat{D}^j_E(F_1)$ in \eqref{e7}
corresponds to the kernel of the projection $J^{2m}(Q)\, \longrightarrow\, J^{2m-j-1}(Q)$
by the isomorphism $f'_{2m}$ in \eqref{ch2}.

Let
\begin{equation}\label{j}
0\, \longrightarrow\, Q\otimes K^{\otimes 2m}_X\, \longrightarrow\,
{\mathbb F}\,=\, J^{2m}(Q)\, \longrightarrow\, J^{2m-1}(Q)\, \longrightarrow\, 0
\end{equation}
be the short exact sequence of jet bundles where,
$\mathbb F$ is identified with $J^{2m}(Q)$ using the isomorphism $f'_{2m}$ in \eqref{ch2}.

As explained before, the connection
$D$ on $E$ need not preserve the subbundle $\mathcal S$ in \eqref{s}. Consider the decomposition
$E\,=\, {\mathbb F}\oplus{\mathcal S}$ in Lemma \ref{lem1}.
Assume that 
\begin{equation}\label{a1}
\widehat{D}_E({\mathcal S})\, =\, {\mathcal S}\oplus (Q\otimes K^{\otimes 2m}_X)
\, \subset\, {\mathcal S}\oplus {\mathbb F}\,=\, E\, ,
\end{equation}
where $Q\otimes K^{\otimes 2m}_X$ is the subbundle of $\mathbb F$ in \eqref{j}, and $\widehat{D}_E
({\mathcal S})\, \subset\, E$ is the holomorphic subbundle given by Lemma \ref{3.1}. Then the second
fundamental form $S(D;{\mathcal S})$ of $\mathcal S$ for the connection $D$ is a holomorphic section 
\begin{equation}\label{sf}
S(D;{\mathcal S})\, \in\, H^0(X,\, \text{Hom}({\mathcal S},\, Q\otimes K^{\otimes (2m+1)}_X))
\end{equation}
$$
\subset\, H^0(X,\, \text{Hom}({\mathcal S},\, {\mathbb F}))\,=\,
H^0(X,\, \text{Hom}({\mathcal S},\,E/{\mathcal S}))\, ;
$$
note that Lemma \ref{lem1} identifies $E/{\mathcal S}$ with ${\mathbb F}$.

\section{Generalized ${\rm SO}(2n,{\mathbb C})$-opers and projective structures}\label{opers}

Through the construction of generalized ${\rm SO}(2n,{\mathbb C})$-quasiopers in Definition 
\ref{def3} and that of generalized $B$-opers in \cite[Definition 2.11]{BSY} we define a 
generalized ${\rm SO}(2n,{\mathbb C})$-oper.

\begin{definition}\label{def4}
A \textit{generalized} ${\rm SO}(2n,{\mathbb C})$--{\it oper} on $X$ is a
generalized $\text{SO}(2n, {\mathbb C})$--quasioper
$(E,\, B_0,\, {\mathcal F}_{\bullet},\, D)$ on $X$ (see Definition \ref{def3})
satisfying the following three conditions:
\begin{enumerate}
\item ${\mathcal S}\,=\, Q\otimes K^{\otimes m}_X$, where $\mathcal S$ and
$Q$ are defined in \eqref{s} and \eqref{e2} respectively,

\item $\widehat{D}({\mathcal S})\, =\, {\mathcal S}\oplus (Q\otimes K^{\otimes 2m}_X)$
(see \eqref{a1} for this condition), and

\item there is a holomorphic section $$\phi\, \in\, H^0(X,\, K^{\otimes (m+1)}_X)$$ such
that the second fundamental form $S(D;{\mathcal S})$ in \eqref{sf} is:
$$
S(D;{\mathcal S})\,=\, \text{Id}_Q\otimes\phi\, .
$$
Note that using the isomorphism ${\mathcal S}\,=\, Q\otimes K^{\otimes m}_X$ in (1), the
second fundamental form $S(D;{\mathcal S})$ in \eqref{sf} is a holomorphic section of
$$
\text{Hom}(Q\otimes K^{\otimes m}_X,\, Q\otimes K^{\otimes (2m+1)}_X)\,=\,
\text{End}(Q)\otimes K^{\otimes (m+1)}_X\, ;
$$
the condition says that this section $S(D;{\mathcal S})$ coincides with $\text{Id}_Q\otimes\phi$.
\end{enumerate}

Two generalized ${\rm SO}(2n,{\mathbb C})$--opers are called \textit{isomorphic} if the
underlying generalized $\text{SO}(2n, {\mathbb C})$--quasiopers are isomorphic (see
Definition \ref{def3}).
\end{definition}

The following lemma is straightforward to prove.

\begin{lemma}\label{lemt2}
Take a holomorphic line bundle $\mathcal L$ on $X$ of order two, and fix a holomorphic
isomorphism $\rho$ as in \eqref{rho}. Let $(E,\, B_0,\, {\mathcal F}_{\bullet},\, D)$ be a generalized
${\rm SO}(2n,{\mathbb C})$--oper on $X$. Then the generalized ${\rm SO}(2n, {\mathbb C})$--quasioper
$(E^1,\, B^1_0,\, {\mathcal F}^1_{\bullet},\, D^1)$ in Lemma \ref{lemt} is also
a generalized ${\rm SO}(2n,{\mathbb C})$--oper.
\end{lemma}

Fix integers $n$ and $m$ as in Definition \ref{def1}; note that $r\,:=\,n/(m+1)$ is an integer, in fact
it is the rank of $F_1$ in \eqref{e-1}. Let
$$
{\mathbb O}_X(n,m)
$$
denote the space of all isomorphism classes of generalized ${\rm SO}(2n,{\mathbb C})$--opers on $X$
of filtration length $2m+1$ (see Definition \ref{def1} and Definition \ref{def4}).

Let
$$
J(X)_2\, \subset\, \text{Pic}^0(X)
$$
be the group of holomorphic line bundles on $X$ of order two; it is isomorphic
to $({\mathbb Z}/2{\mathbb Z})^{\oplus 2g}$, where $g\,=\, \text{genus}(X)$.

Let
\begin{equation}\label{cx}
{\mathcal C}_X
\end{equation}
be the space of all isomorphism classes of holomorphic
$\text{SO}(r, {\mathbb C})$--bundles on $X$ equipped with a holomorphic connection. So ${\mathcal C}_X$
in \eqref{cx} parametrizes isomorphism classes of pairs $(V,\, B_V)$, where $V$ is a holomorphic vector bundle
on $X$ of rank $r$ with $\bigwedge^r V\,=\, {\mathcal O}_X$, and
$B_V\, \in\, H^0(X, \, \text{Sym}^2(V^*))$
is a fiberwise nondegenerate symmetric bilinear form on $V$. We recall that a holomorphic connection on
$(V,\, B_V)$ is a holomorphic connection $D_V$ on $V$ such that
$$
\partial B_V(s,\, t)\,=\, B_V(D_V(s),\, t) + B_V(s,\, D_V(t))
$$
for all locally defined holomorphic sections $s$ and $t$ of $V$. Let
$$
{\mathfrak P}(X)
$$
be the space of all projective structures on $X$; see \cite{Gu}, \cite{Bi1} for projective
structures on $X$. Then, one has the following correspondence between generalized ${\rm SO}(2n,
{\mathbb C})$--opers and geometric structures.

\begin{theorem}\label{thm1}
First assume that the integer $r\,=\,n/(m+1)$ is odd. 
There is a canonical bijection between ${\mathbb O}_X(n,m)$ and the Cartesian product
$$
{\mathcal C}_X\times {\mathfrak P}(X)\times \left(H^0(X,\, K^{\otimes (m+1)}_X)\oplus
\left(\bigoplus_{i=2}^m H^0(X,\, K^{\otimes 2i}_X)\right)\right)\times J(X)_2\ .
$$

If $r$ is even, then
there is a canonical bijection between ${\mathbb O}_X(n,m)$ and the Cartesian product
$$
{\mathcal C}_X\times {\mathfrak P}(X)\times \left(H^0(X,\, K^{\otimes (m+1)}_X)\oplus
\left(\bigoplus_{i=2}^m H^0(X,\, K^{\otimes 2i}_X)\right)\right)\ .
$$
\end{theorem}

\begin{proof}
Assume that $r\,=\,n/(m+1)$ is an odd integer. Take an element
\begin{equation}\label{te}
(\alpha,\, \beta,\, \gamma,\, \delta,\, {\mathcal L})
\end{equation}
in 
$$
{\mathcal C}_X\times {\mathfrak P}(X)\times \left(H^0(X,\, K^{\otimes (m+1)}_X)\oplus
\left(\bigoplus_{i=2}^m H^0(X,\, K^{\otimes 2i}_X)\right)\right)\times J(X)_2\, ,
$$
such that \begin{itemize}
\item $\alpha\, =\, (V,\, B_V,\, D_V)\, \in\, {\mathcal C}_X$, where $(V,\, B_V)$
is a holomorphic $\text{SO}(r,{\mathbb C})$--bundle
on $X$ equipped with a holomorphic connection $D_V$,

\item $\beta$ is a projective structure on $X$,

\item $\gamma$ is a holomorphic section
\begin{equation}\label{g}
\gamma\,\in\, H^0(X,\, K^{\otimes (m+1)}_X)\, ,
\end{equation}

\item $\delta\,\in\,
\bigoplus_{i=2}^m H^0(X,\, K^{\otimes 2i}_X)$, and

\item $\mathcal L$ is a holomorphic line bundle on $X$ of order two.
\end{itemize}
Using \cite[Theorem 4.6]{BSY}, the triple
$(\alpha,\, \beta,\, \delta)$ produces the following:
\begin{itemize}
\item a nondegenerate holomorphic symmetric bilinear form $B_1$ on $J^{2m}(V\otimes K^{-\otimes m}_X)$, and

\item a holomorphic connection $\textbf{D}_1$ on $J^{2m}(V\otimes K^{-\otimes m}_X)$ that preserves the
bilinear form $B_1$.
\end{itemize}
Furthermore, the triple $(J^{2m}(V\otimes K^{-\otimes m}_X),\, B_1,\, \textbf{D}_1)$, together with the
filtration of \linebreak $J^{2m}(V\otimes K^{-\otimes m}_X)$ given by 
\begin{equation}\label{f1}
0\, =\, A_0\, \subset\, A_1\, \subset\, A_2 \, \subset\, \cdots \, \subset\, A_{2m} \, \subset\,
A_{2m+1}\,=\, J^{2m}(V\otimes K^{-\otimes m}_X)\, ,
\end{equation}
where $A_i$ is the kernel of the natural projection $J^{2m}(V\otimes K^{-\otimes m}_X)\, \longrightarrow\,
J^{2m-i}(V\otimes K^{-\otimes m}_X)$, define a generalized $B$--oper (see \cite[Definition 2.11]{BSY}
and \cite[Theorem 4.6]{BSY}).

Now consider the holomorphic vector bundle $$E\, :=\, J^{2m}(V\otimes K^{-\otimes m}_X)\oplus V$$ on $X$.
Note that it is equipped with nondegenerate holomorphic symmetric bilinear form $B_1\oplus B_V$. The
holomorphic connection $\textbf{D}_1$ on $J^{2m}(V\otimes K^{-\otimes m}_X)$ and the holomorphic
connection $D_V$ on $V$ together produce the holomorphic connection $\textbf{D}_1\oplus D_V$ on $
J^{2m}(V\otimes K^{-\otimes m}_X)\oplus V\,=\, E$. This
connection $\textbf{D}_1\oplus D_V$ on $E$ evidently preserves the bilinear form $B_1\oplus B_V$ on $E$.

Using the holomorphic connection $\textbf{D}_1\oplus D_V$ on $E$
and the section $\gamma$ in \eqref{g}, we shall construct another holomorphic connection on $E$. Since
$E\, =\, J^{2m}(V\otimes K^{-\otimes m}_X)\oplus V$, using the filtration in \eqref{f1}, we have
\begin{equation}\label{z1}
\text{Hom}(V,\, V\otimes K^{\otimes m}_X)\,=\,\text{Hom}(V,\, A_1)\, \subset\,
\text{Hom}(V,\, J^{2m}(V\otimes K^{-\otimes m}_X))\, ;
\end{equation}
in \eqref{f1}, note that $A_1\,=\, V\otimes K^{\otimes m}_X$. Similarly, we have
\begin{equation}\label{z2}
\text{Hom}(V\otimes K^{-\otimes m}_X,\, V)\,=\,
\text{Hom}(A_{2m+1}/A_{2m},\, V)\, \subset\,
\text{Hom}(J^{2m}(V\otimes K^{-\otimes m}_X),\, V)\, ;
\end{equation}
in \eqref{f1}, note that $$A_{2m+1}/A_{2m}\,=\, V\otimes K^{-\otimes m}_X\, ,$$ so
the quotient map $A_{2m+1}\, \longrightarrow\, A_{2m+1}/A_{2m}$ produces the inclusion
map $$\text{Hom}(A_{2m+1}/A_{2m},\, V)\, \hookrightarrow\,
\text{Hom}(J^{2m}(V\otimes K^{-\otimes m}_X),\, V)\, .$$ On the other hand,
$$
\text{Hom}(V,\, J^{2m}(V\otimes K^{-\otimes m}_X))\oplus \text{Hom}(J^{2m}(V\otimes K^{-\otimes m}_X),\, V)
$$
$$
\subset\, \text{End}(J^{2m}(V\otimes K^{-\otimes m}_X)\oplus V)\,=\, \text{End}(E)\, .
$$
Hence from \eqref{z1}, \eqref{z2} we conclude that
\begin{equation}\label{i1}
(\text{End}(V)\otimes K^{\otimes m}_X)\oplus (\text{End}(V)\otimes K^{\otimes m}_X)
\,=\,\text{Hom}(V,\, V\otimes K^{\otimes m}_X)\oplus\text{Hom}(V\otimes K^{-\otimes m}_X,\, V)
\end{equation}
$$
\subset\, \text{Hom}(V,\, J^{2m}(V\otimes K^{-\otimes m}_X))\oplus \text{Hom}(J^{2m}(V\otimes K^{-\otimes m}_X),\, V)
\, \subset\, \text{End}(E)\, .
$$
From \eqref{i1} it follows immediately that
\begin{equation}\label{i2}
(\text{Id}_V\otimes\gamma,\, -\text{Id}_V\otimes\gamma)\, \in\, H^0(X, \, \text{End}(E)\otimes K_X)\, ,
\end{equation}
where $\gamma$ is the section in \eqref{g}.

Any two holomorphic connections on the holomorphic vector bundle $E$ differ by a holomorphic
section of $\text{End}(E)\otimes K_X$. Since $\textbf{D}_1\oplus D_V$ is a holomorphic connection on $E$, from
\eqref{i2} we conclude that
\begin{equation}\label{de}
\textbf{D}_E\,:=\, (\textbf{D}_1\oplus D_V)+(\text{Id}_V\otimes\gamma,\, -\text{Id}_V\otimes\gamma)
\end{equation}
is a holomorphic connections on the holomorphic vector bundle $E$. Since the connection $\textbf{D}_1\oplus D_V$
on $E$ preserves the bilinear form $B_1\oplus B_V$ on $E$, from the construction of $(\text{Id}_V
\otimes\gamma,\, -\text{Id}_V\otimes\gamma)\, \in\, H^0(X, \text{End}(E)\otimes K_X)$ in \eqref{i2} it follows
that the connection $\textbf{D}_E$ on $E$ in \eqref{de} also preserves the bilinear form $B_1\oplus B_V$ on $E$.

Using the filtration of $J^{2m}(V\otimes K^{-\otimes m}_X)$ in \eqref{f1} we shall construct
a filtration of holomorphic subbundles on $E$. Let
\begin{equation}\label{f2}
0\, =\, A'_0\, \subset\, A'_1\, \subset\, A'_2 \, \subset\, \cdots \, \subset\, A'_{2m} \, \subset\,
A'_{2m+1}\,=\, E\,=\, J^{2m}(V\otimes K^{-\otimes m}_X)\oplus V
\end{equation}
be the filtration where, $A'_i\,=\, A_i\oplus 0$ for all $0\, \leq\, i\,\leq\, m$ and
$A'_i\,=\, A_i\oplus V$ for all $m+1\, \leq\, i\,\leq\, 2m+1$.

From the above, we have that the holomorphic vector bundle $E$,
the bilinear form $B_1\oplus B_V$,
the filtration $\{A'_i\}_{i=0}^{2m+1}$ in \eqref{f2}, and
 the holomorphic connection $\textbf{D}_E$ in \eqref{de}
together
define a generalized ${\rm SO}(2n,{\mathbb C})$--oper.

In view of Lemma \ref{lemt2}, the above generalized ${\rm SO}(2n,{\mathbb C})$--oper
$$(E,\, B_1\oplus B_V,\, \{A'_i\}_{i=0}^{2m+1},\, \textbf{D}_E)$$ and the line bundle $\mathcal L$ in \eqref{te}
together produce a generalized ${\rm SO}(2n,{\mathbb C})$--oper.
It is evident that this generalized ${\rm SO}(2n,
{\mathbb C})$--oper is an element of ${\mathbb O}_X(n,m)$.

Now assume that the integer $r$ is even. Let $V$ be a holomorphic vector bundle on $X$ of rank $r$, and let $B_V\, 
\in\, H^0(X, \, \text{Sym}^2(V^*))$ is a fiberwise nondegenerate symmetric bilinear form on $V$. Then we have 
$\bigwedge^r V\,=\, {\mathcal O}_X$, because $r$ is even. Therefore, if $(V,\, B_V)$ is an holomorphic 
$\text{SO}(r,{\mathbb C})$--bundle, then for any ${\mathcal L}\, \in\, J(X)_2$, that pair $(V\otimes{\mathcal 
L},\, B_V\otimes\rho)$, where $\rho\, :\, {\mathcal L}^{\otimes 2} \, \longrightarrow\, {\mathcal O}_X$ is an 
isomorphism (as in \eqref{rho}), is again a holomorphic $\text{SO}(r,{\mathbb C})$--bundle. Hence in the case of 
even $r$, when we consider ${\mathcal C}_X$, tensoring with line bundles of order two are already taken into 
account, so we no longer need to take line bundles of order two separately (which was needed in the previous case 
of $r$ being odd).

Therefore, the above constructions identify
${\mathbb O}_X(n,m)$ with
$$
{\mathcal C}_X\times {\mathfrak P}(X)\times \left(H^0(X,\, K^{\otimes (m+1)}_X)\oplus
\left(\bigoplus_{i=2}^m H^0(X,\, K^{\otimes 2i}_X)\right)\right)\, .
$$

We shall now describe the reverse construction. Again first assume that the integer $r$ is odd.

Let $$(E,\, B_0,\, {\mathcal F}_{\bullet},\, D)\, \in\,
{\mathbb O}_X(n,m)$$ be a generalized ${\rm SO}(2n,{\mathbb C})$--oper on $X$.
Consider the decomposition
$$
{\mathbb F}\oplus{\mathcal S}\, =\, E
$$
in Lemma \ref{lem1}. As noted in \eqref{B}, we have that 
$$
B_0\,=\, B_{\mathbb F}\oplus B_{\mathcal S}\, .
$$
Moreover, from Corollary \ref{cor2}, the connection $D^{\mathbb F}$ (respectively, 
$D^{\mathcal S}$) on $\mathbb F$ (respectively, $\mathcal S$) preserves the bilinear form 
$B_{\mathbb F}$ (respectively, $B_{\mathcal S}$). The vector bundle $\mathbb F$ has a 
filtration of holomorphic subbundles (see \eqref{e7}), which we shall denote by $\widetilde{\mathcal 
F}_\bullet$. Recall that the isomorphism $f'_{2m}$ in \eqref{ch2} takes the
filtration $\widetilde{\mathcal F}_\bullet$ to the filtration of $J^{2m}(Q)$ given by the
exact sequences in \eqref{js}.

Note that $({\mathbb F},\, B_{\mathbb F},\, \widetilde{\mathcal F}_\bullet,\, D^{\mathbb F})$ satisfies all the conditions
needed to define a generalized $B$--oper (see \cite[Definition 2.11]{BSY}) except possibly the only condition
$$
\det {\mathbb F}\,=\, {\mathcal O}_X\, .
$$
In any case,
\begin{equation}\label{dt}
{\mathcal L}\, :=\, \det {\mathbb F}\, \in\, J(X)_2\, .
\end{equation}
For the vector bundle ${\mathbb F}'\, :=\, {\mathbb F}\otimes {\mathcal L}$, we have
$\det {\mathbb F}'\,=\, {\mathcal O}_X$.

Let $\rho\, :\, {\mathcal L}^{\otimes 2}
\, \longrightarrow\, {\mathcal O}_X$ be an isomorphism (as in \eqref{rho}). Define
the nondegenerate symmetric bilinear form
$$
B'_{\mathbb F}\, :=\, B_{\mathbb F}\otimes\rho
$$
on ${\mathbb F}'\, =\, {\mathbb F}\otimes {\mathcal L}$. Tensoring the above filtration $\widetilde{\mathcal
F}_\bullet$ of ${\mathbb F}$ by $\mathcal L$ we get a filtration of holomorphic subbundles of
${\mathbb F}'$; this filtration of ${\mathbb F}'$ will be denoted by $\widetilde{\mathcal
F}'_\bullet$. The holomorphic connection $D^{\mathbb F}$ on $\mathcal F$ and
the canonical connection $D^{\mathcal L}$ on ${\mathcal L}$ in \eqref{dl} together define a holomorphic
connection
$$
D'_{\mathbb F}\, :=\, D^{\mathbb F}\otimes\text{Id}_{\mathcal L}+ \text{Id}_{\mathbb F}\otimes D^{\mathcal L}
$$
on ${\mathbb F}'$ (as done in \eqref{dl2}).

Now $({\mathbb F}',\, B'_{\mathbb F},\, \widetilde{\mathcal F}'_\bullet,\, D'_{\mathbb F})$
is a generalized $B$--oper \cite{BSY}. Therefore, from \cite[Theorem 4.6]{BSY}
we obtain a triple
\begin{equation}\label{t}
(\alpha,\, \beta,\, \delta)\, \in\, {\mathcal C}_X\times
{\mathfrak P}(X)\times \left(\bigoplus_{i=2}^m H^0(X,\, K^{\otimes 2i}_X)\right)
\end{equation}
associated to $({\mathbb F}',\, B'_{\mathbb F},\, \widetilde{\mathcal F}'_\bullet,\, D'_{\mathbb F})$.

Next consider the second fundamental form for the subbundle
${\mathcal S}\, \subset\, E\, \,=\, {\mathbb F}\oplus{\mathcal S}$ for the connection $D$
on $E$. Let
$$
S(D;{\mathcal S})\, \in\, H^0(X,\, \text{Hom}({\mathcal S},\, {\mathbb F})\otimes K_X)
$$
be the second fundamental form for the subbundle $\mathcal S$ for the connection $D$ on
$E$. From Definition \ref{def4} and \eqref{sf} we know that
$$
S(D;{\mathcal S})\, \in\,H^0(X,\, \text{Hom}({\mathcal S},\, F_1)\otimes K_X)
\,=\, H^0(X,\, \text{Hom}({\mathcal S},\, Q\otimes K^{\otimes (2m+1)}_X))
$$
$$
\subset\, H^0(X,\, \text{Hom}({\mathcal S},\, {\mathbb F})\otimes K_X)\, ;
$$
we note that $F_1\,=\, Q\otimes K^{\otimes 2m}_X$; this follows from the fact that 
isomorphism $f'_{2m}$ in \eqref{ch2} takes the filtration ${\mathcal F}'_\bullet$ to the 
filtration of $J^{2m}(Q)$ given by the exact sequences in \eqref{js}. Since
${\mathcal S}\,=\, Q\otimes K^{\otimes m}_X$ (see Definition \ref{def4}), we have
$$
S(D;{\mathcal S})\, \in\, H^0(X,\, \text{End}({\mathcal S}) \otimes K^{\otimes (m+1)}_X)\, .
$$
We recall from Definition \ref{def4} that 
$
S(D;{\mathcal S})\,=\, \text{Id}_{\mathcal S}\otimes \phi\, ,
$
where $\phi\, \in\, H^0(X,\, K^{\otimes (m+1)}_X)$.
Then, we have
$$
(\alpha,\, \beta,\, \phi,\, \delta,\, {\mathcal L})\, \in\, 
{\mathcal C}_X\times {\mathfrak P}(X)\times \left(H^0(X,\, K^{\otimes (m+1)}_X)\oplus
\left(\bigoplus_{i=2}^m H^0(X,\, K^{\otimes 2i}_X)\right)\right)\times J(X)_2\ ,
$$
where $(\alpha,\, \beta,\, \delta)$ is constructed in \eqref{t} and
${\mathcal L}$ is the line bundle in \eqref{dt}.
It is straightforward to check that the two constructions are inverses of each other.

When the integer $r$ is even, the above reverse construction is simpler because in that case
$({\mathbb F},\, B_{\mathbb F},\, \widetilde{\mathcal F}_\bullet,\, D^{\mathbb F})$ is already
a generalized $B$--oper, so the construction of
$({\mathbb F}',\, B'_{\mathbb F},\, \widetilde{\mathcal F}'_\bullet,\, D'_{\mathbb F})$ from it is not
needed.
\end{proof}

\section*{Acknowledgements}

IB is supported by a J. C. Bose Fellowship. LPS is partially supported by the
NSF CAREER Award DMS-1749013.

%%%%%%%%%%%%%%%%%%%%%%%%%%%%%%%%%%%%%%%%%%%%%%%%%%%%%%%%%%%%%%%%

\end{document}